\documentclass[11pt, reqno]{amsart}
\usepackage[a4paper, total={5.35 in, 8.3 in}]{geometry}
\usepackage{amsmath,amsfonts,amssymb}
\usepackage{amsfonts}
\usepackage{verbatim}
\usepackage{enumerate}
\usepackage{amsthm}
\usepackage{url}
\usepackage[colorlinks]{hyperref}
\usepackage{tikz}

\usepackage[initials]{amsrefs}
\makeatletter
\@namedef{subjclassname@2010}{%
  \textup{2010} Mathematics Subject Classification}
\makeatother

\newtheorem{theorem}{Theorem}[section]
\newtheorem{corollary}[theorem]{Corollary}
\newtheorem{lemma}[theorem]{Lemma}
\newtheorem{proposition}[theorem]{Proposition}

\theoremstyle{definition}

\newtheorem*{remark}{Remark}

\numberwithin{equation}{section}

\newcommand{\F}{\mathbb{F}}
\newcommand{\Fn}{\mathbb{F}_{q^n}}
\DeclareMathOperator{\Tr}{Tr}
\DeclareMathOperator{\ord}{ord}

\newcommand{\Q}{\mathcal{Q}}

\newcommand{\X}{\mathcal{X}}
\newcommand{\Y}{\mathcal{Y}}
\newcommand{\Z}{\mathcal{Z}}
\newcommand{\N}{\mathcal{N}}
\newcommand{\M}{\mathcal{M}}

\begin{document}

\baselineskip=17pt


\title {The trace of  primitive and  $2$-primitive elements in finite fields, revisited}
 \renewcommand\rightmark  {THE TRACE  OF PRIMITIVE AND   $2$-PRIMITIVE ELEMENTS, REVISITED}
\author[S.D.~Cohen]{Stephen D. Cohen } \thanks{The first author is Emeritus Professor of Number Theory, University of Glasgow}

\address{6 Bracken Road, Portlethen, Aberdeen AB12 4TA, Scotland, UK}
\email{Stephen.Cohen@glasgow.ac.uk}
\author[G. Kapetanakis]{Giorgos Kapetanakis}
\address{Department of Mathematics, University of Thessaly, 3rd km. Old National Road Lamia-Athens, 35100 Lamia, Greece}
\email{gnkapet@gmail.com}

\date{\today}

\begin{abstract}
By definition primitive and $2$-primitive elements of a finite field extension $\Fn$ have order $q^n-1$ and $(q^n-1)/2$, respectively.  We have already shown  that, with minor reservations,  there exists a primitive element and a $2$-primitive element $\xi \in \Fn$ with prescribed trace in the ground field $\F_q$.  Here we amend our previous proofs of these results, firstly,  by a reduction of these problems to extensions of prime degree $n$ and,  secondly,    by deriving an exact expression for the number of squares in $\Fn$ whose trace has prescribed value in  $\F_q$. The latter corrects an error in the  proof in the case of  $2$-primitive elements. We also streamline the necessary  computations.

\end{abstract}



\maketitle

\section{Introduction}

Let $q$ be the power of a prime $p$ and $n\geq  2$ be an integer.  Denote by $\F_q$  the finite field  of  order $q$ and by $\Fn$ its extension of degree $n$. 
A  primitive element of  $\Fn$ is a generator of its (cyclic)  mutiplicative group.  More generally, for a divisor $r$ of $q^n-1$, an $r$-primitive element has been defined as an element of order $(q^n-1)/r$. In this sense, a $1$-primitive element has the same meaning as a primitive element. We shall be concerned solely with primitive and $2$-primitive elements here.

We denote by $\Tr$ the trace function $\F_{q^n}\to\F_q$, that is
\[
\Tr(\xi) := \sum_{i=0}^{n-1} \xi^{q^i},\ \xi\in\F_{q^n} .
\]
By  {\em the trace problem for $r$-primitive elements}  will be meant a study of whether or not,   given $\beta \in \F_q$, there exists an $r$-primitive $\xi \in \Fn$ such that $\Tr(\xi)=\beta$.
The resolution of the trace problem for primitive elements in  \cite{cohen90} is a fundamental result.  Indeed,   in a  recent monograph \cite{hachenbergerjungnickel20},  Hachenberger and Jungnickel  devote their final chapter  to it  (see Result 14.1.1).
\begin{theorem}\label{thm:prim_trace}
Let $q$ be a prime power, $n$ an integer and $\beta\in\F_q$. Unless $(n,\beta) = (2,0)$ or $(n,q)=(3,4)$, there exists a primitive $\xi\in\F_{q^n}$ with $\Tr(\xi)=\beta$.
\end{theorem}
The original proof of  Theorem~\ref{thm:prim_trace} was obtained by assembling components from different sources and implicitly involved some direct verification by a computer.
For this reason another proof was given in \cite{cohenpresern05} which could be checked theoretically with the aid of a basic calculator. The solution in \cite{hachenbergerjungnickel20} contains additional  algebraic ideas. It too  reduces the level of computation required to a minimum.

More recently, the authors  resolved the basic trace problem for  $2$-primitive elements in $\Fn$ (necessarily for odd prime powers $q$), \cite{cohenkapetanakis20}, Theorem~1.3.
\begin{theorem}\label{thm:main}
Let $q$ be an odd prime power.
\begin{enumerate}
\item Let $\beta \in\F_q$. Suppose $n \geq 3$. Then there exists a  $2$-primitive element $\xi$ of $\F_{q^n}$ such that $\Tr(\xi)=\beta$.
\item  Suppose $n=2$.    Let $\beta \in\F_q^*$. There exists some $2$-primitive element $\xi$ of $\F_{q^2}$ such that $\Tr(\xi)=\beta$, unless $q=3,5,7,9,11,13$ or $31$.   Further, in the exceptional case when $q=9$ and $b\in \{\pm1,\pm i\}$ (where $i^2=-1$), there exists a $2$-primitive element $\xi$  in $\F_{81}$ with $\Tr(\xi) =\beta$. 
\end{enumerate}
\end{theorem}
 \begin{remark}
 When $n=2$, if $\xi \in \F_{q^2}$ has trace 0, then its order must be at most $2(q-1)$ and so $\xi$ cannot be $2$-primitive unless $q=3$ in which case $\F_{q^2}^* =\F_9^*= \{\pm 1, \pm i, \pm 1 \pm i\}$.  Here the primitive elements $\pm1\pm i$ have trace $\pm1$ and the $2$-primitive elements $\pm i$ do have trace $0$.  From now on, if $n=2$, assume $\beta \neq 0$.
 \end{remark}
At this point we need to admit that that the argument in the proof of Theorem~1.3  in \cite{cohenkapetanakis20}  relating to the terms in the expression (4.1) involving  the principal character $\chi_1$ contain an error which exaggerates the influence of these terms and is pervasive.    In particular, the remark at the end of Section~4.1 raises a problem that does not actually arise.  Because it is not easy to provide a brief corrigendum to rectify this flaw,  one aim of the present article is to purify the exposition.

Another important contribution to the analysis of the  trace problem is its  reduction to the study of extensions of prime degree $n$.  This forms the substance of Section~\ref{sec:primedegree}.  A subsdiary result which, given $\beta \in \F_q$, yields an exact expression for the number of  squares $\xi \in \Fn$ for which $\Tr(\xi)=(\beta)$ is provided in Section~\ref{squares}.
This eliminates the flaw in the original proof.    A further feature of this exposition is  more efficient working in the  calculations necessary to apply the theory.
\section{Reduction to extensions of prime degree}\label{sec:primedegree}
Here we show how the proofs of Theorems \ref{thm:prim_trace} and \ref{thm:main} can be deduced from the case in which $n$ is a prime.   In this section, given $d|n$ we denote by $\Tr_{n/d}$ the trace function from $\F_{q^n}$ to $\F_{q^d}$. The following lemma is also related to Proposition~14.2.2 in \cite{hachenbergerjungnickel20}
\begin{lemma}\label{induction} Suppose that  Theorem~\ref{thm:prim_trace} has  been established for all prime values of $n$ and all  prime powers $q$. Then  Theorem~\ref{thm:prim_trace} holds for arbitrary values of $n$ and $q$

Similarly, suppose Theorem~\ref{thm:main}   has been established  for all prime values of $n$ and  odd  prime powers $q$. Then Theorem~\ref{thm:main}  holds for arbitrary values of $n$ and $q$.
\begin{proof} Let $t$ denote the total  number of primes (including their mutiplicity) in the prime decomposition of $n$.  The proof is by induction on $t$ with the case of  $t=1$ corresponding to the stated assumption.

For the induction step write $n=\ell m$, where $\ell$ is any prime dividing $n$ and $1<m<n$.  Given $\beta \in \F_q$, by elementary linear algebra, it is evident that there are $q^{\ell-1}$ elements in $\F_{q^\ell}$ whose trace in $\F_q$ is $\beta$.  Hence, in every case, even if $\beta=0$, we can choose a {\em non-zero} $\alpha \in \F_{q^\ell}$  such that $\Tr_{\ell/1}(\alpha)=\beta$.  Next, for this non-zero  element $\alpha$  apply the induction hypothesis to the extension $\F_{q^n}=\F_{q^{\ell m}}$ over $\F_{q^l}$.  Now,  the exceptional cases in the theorems can only be relevant if  $q^\ell$  is a prime (which it is not) or, in Theorem~\ref{thm:main}, when $q^\ell=9$  and $m=2$.   Moreover, this latter situation  can occur only if   $q=3, \ell=2$  and $n=4$ (so that $t=2$).   We deduce that, with $r=1$ or $2$ respectively,  there exists an $r$-primitive element  $\xi \in \F_{q^n}$ such that $\Tr_{n/\ell}(\xi)=\alpha$, with the possible exception of  the case when $r=2$, $t=2$,  $q=3$,  $n=4, \ell=2$ and $\alpha= \pm 1$.  In this last situation, the only possible non-zero values for $\beta \in \F_q=\F_3$  are $\pm 1$  and we choose $\alpha = \mp 1$, respectively (since
$\Tr_{2/1}(\mp1)= \pm 1$, respectively).   But, from the final assertion of Theorem~\ref{thm:main}, there exists    $2$-primitive elements  $\xi_1, \xi_{-1} \in \F_{81}$  with   $\Tr_{4/2}(\xi_1)=1$ and $\Tr_{4/2}(\xi_{-1}) =-1$.     Finally, more generally $\Tr(\xi)=\Tr_{\ell/1}(\Tr_{n/\ell}(\xi))=\Tr_{\ell/1}(\alpha)=\beta$ in every case  and the lemma follows by induction.
\end{proof}
\end{lemma}
Lemma~\ref{induction} significantly reduces the calculations in the previous proof.  In particular, in the trace problem for primitive elements, it eliminates the discussion  in \cite{cohenpresern05} of the case whe $n=4$ in Section~4.1 and with respect to the case $n=4,6$ in Section~5.  Similarly, in \cite{hachenbergerjungnickel20}, although the authors choose to perform calulations related to degrees $n=4, 6$ in Sections 14.5, 14.6,   there would be no need for these to establish simple existence.
This completes our discussion of the trace property for primitive elements.  We now focus on the property for $2$-primitive elements and, in particular, assume that $q$ is odd.

By Lemma~\ref{induction}  the material (including the calculations)  in Section~5.2 of \cite{cohenkapetanakis20} is now redundant.   Further, we  assume that $n$ is prime and thus consider only pairs $(q,n)$ where $q$ is an odd prime power and $n$ is prime.  Indeed, we shall additionally sift out pairs $(q,n)$ which are {\em odd} in the sense that $(q^n-1)/2$ is odd, which means that $n$  is an odd prime and $q \equiv 3 \mod 4$.

\begin{lemma}\label{oddqn}
Suppose $(q,n)$  is odd. Then $\xi \in \F_{q^n}$ is $2$-primitive if and only if $-\xi$ is primitive.
\end{lemma}
\begin{proof}
We have that $\frac{q^n-1}{2}$ is odd, so $\xi$ is $2$-primitive if and only if $\xi$ is both $\frac{q^n-1}{2}$-free and a square in   $\F_{q^n}$. Now $(-1)^{(q^n-1)/2}=-1$, thus $-1$ is a nonsquare in $\F_{q^n}$.   Hence $\xi \in \F_{q^n}$ is a nonsquare  if and only if $-\xi$ is a square.
Moreover, $\xi$ is $\frac{q^n-1}{2}$-free if and only if $-\xi$ is $\frac{q^n-1}{2}$-free.  The result follows.
\end{proof}

It follows from Lemma~\ref{oddqn} that when $(q,n)$ is odd, the number of $2$-primitive elements in $\F_{q^n}$ is the same as the number of primitive elements (namely $\phi(q^n-1)$, where $\phi$ is Euler's function).
In the situation of Lemma~\ref{oddqn} we can deduce the existence theorem for $2$-primitive elements from Theorem~\ref{thm:prim_trace}.
\begin{lemma}\label{oddthm}
Suppose $(q,n)$ is odd.  Then, given arbitrary $\beta \in \F_q$, there exists a $2$-primitive element of $\F_{q^n}$ with trace $\beta$.  (We can describe this in other words by saying that   $(q,n)$ {\em has the trace property for $2$-primitive elements}.)
\end{lemma}
\begin{proof}
From Theorem~\ref{thm:prim_trace}, there exists a primitive element $\xi$ of $\F_{q^n}$ with trace $-\beta$.  By Lemma~\ref{oddqn}, $-\xi$ is $2$-primitive and $\Tr(-\xi)=-\Tr(\xi)=\beta$.
\end{proof}
\begin{remark}
Since Theorem~\ref{thm:prim_trace} was established in \cite{cohenpresern05}, without recourse to direct verification for any pair $(q,n)$, then the same can be said for Theorem~\ref{oddthm}.               
\end{remark}
From now on we assume  that $(q,n)$ is \emph{even}, i.e., that $(q^n-1)/2$  is even,  in which case $(-1)^{(q^n-1)/2}=1$ and so $-1$ is a square in $\F_{q^n}$.   This implies that, either $n=2$  or $n$ is odd and the prime $q \equiv 1 \mod 4$.  Further,  $\xi \in \F_{q^n}$ is $2$-primitive if and only if $-\xi$ is  $2$-primitive.    In this situation, a $2$-primitive element can be viewed simply as the square $\xi^2$ of a primitive element $\xi $.  Hence, our problem is to confirm that there exists a primitive element $\xi \in \F_{q^n}$ for which $\Tr(\xi^2) =\beta$.  Observe that, if $\xi$ is primitive, then both $(\pm \xi)^2$ yield the same $2$-primitive element $\xi^2$.  In particular, it is clear that the the total number of $2$-primitive elements in $\F_{q^n}$ is $\frac{\phi(q^n-1)}{2}$.
%
  

\section{The number of squares with prescribed trace}\label{squares}
At this point it is convenient to derive, using additive characters,  an explicit expresion for the number of non-zero squares  $\xi \in \Fn$  for which $\Tr(\xi)=\beta$  for any given $\beta \in \F_q$.  We continue to assume $n$ is prime and $(q,n)$ is even (although these restrictions could be lifted).  

We employ additive characters of both $\F_{q^n}$ and $\F_q$.
Let $\psi$ be the canonical additive character of $\F_q$, that is $\psi (g) = \exp (2\pi i\Tr_0(g)/p)$, where $\Tr_0$ stands for the {\em absolute trace} of $g\in\F_q$, i.e., its trace over $\F_p$, the prime subfield of $\F_q$.
Then an arbitrary additive character of $\F_q$ has the action which takes $g\in\F_q$ onto $\psi (ug)$ and thereby, as $u$ varies, we obtain all the $q$ additive characters of $\F_q$,
whose set we will denote by $\widehat{\F_q}$.  For the trivial character, take $u=0$.  Then  the characteristic function for elements of $\F_{q^n}$, with trace $\beta$, can be expressed as
\begin{equation}\label{chartrace}
t_\beta (\xi) := \frac{1}{q} \sum_{u\in\F_q} \bar\psi(u\beta) \tilde\psi(u\xi) ,
\end{equation}
where, $\bar\psi$ stands for the inverse of $\psi$ and $\tilde\psi$ stands for the \emph{lift} of $\psi$ to an additive character of $\F_{q^n}$, i.e., for every $\xi\in\F_{q^n}$, we have that $\tilde\psi(\xi) = \psi(\Tr(\xi))$. In particular, $\tilde\psi$ is the canonical character of $\F_{q^n}$.

  The next two lemmas recall  standard facts about the quadratic Gauss  sum over over $\Fn$  (see \cite{hua82}, Section~7.5, Theorem~5.4, \cite{lidlniederreiter97}, Section~5.2).  
\begin{lemma}\label{gauss}
 Let  $u \in \F_{q^n}$ and set  $$g_n(u)=\sum _{\xi \in \F_{q^n}}\psi(u\xi^2)= \sum_{\xi \in \F_{q^n}}\chi_2(\xi)\psi(u\xi),$$ where $\chi_2$ denotes the quadratic character and $\psi$ the canonical additive character on $\F_{q^n}$. Then  $$g_n(u)= \chi_2(u)g_n(1).$$
  \end{lemma}
 \begin{lemma}\label{lem:A} Let $q$ be a power of the prime $p$ and $u \in \Fn$. 
  
  If  $n=2$, then $g_2(u)  =\chi_2(u)\varepsilon_1 q$,  where
  \begin{equation}\label{eps1}
 \varepsilon_1= \begin{cases}\ \ 1, & \text{if } q \equiv 3 \  (\mathrm{mod}\ 4), \\
  -1, &  \text{if } q \equiv 1\  (\mathrm{mod}\ 4).
\end{cases} 
\end{equation}
 On the other hand, if   $n$ is an odd prime, then
$ g_n(u)= \chi_2(u)\varepsilon_2 q^{n/2}$,
where 
\begin{equation}\label{eps2}
 \varepsilon_2= \begin{cases}\ \ 1, & \text{if }  p \equiv 1\   (\mathrm{mod}\ 4) \text{ and }   q \text{ is a nonsquare},\\
 &    \text{or }   p \equiv 3 \  (\mathrm{mod}\ 4) \text{ and } q \text{ is a square but not a $4$th power},\\
 -1, &   \text{if }  p \equiv  1 \  (\mathrm{mod}\ 4)   \text{ and }   q \text{ is a square},\\
 &    \text{or }   p \equiv 3 \  (\mathrm{mod}\ 4)  \text{ and } q \text{ is  a $4$th power}. 
\end{cases} 
\end{equation}
\end{lemma}

  \begin{proposition}\label{m=1}
  Assume $(q,n$) is even and $n$ is a prime. Given $\beta \in \F_q$, let  $\M_\beta$  be the number  non-zero {\em squares} $\xi \in  \F_{q^n}$ with $\Tr(\xi)=\beta$. 
  
   If $n=2$  then,  for $\beta \in \F_q^*$, 
 \begin{equation}\label{eq:M1} 
 \M_\beta=\frac{1}{2}(q-\varepsilon_1),
    \end{equation}
   where \[\varepsilon_1=  
   \begin{cases}\;\;1,  &  \text{if } q \equiv 3 \mod 4\\
   -1, &  \text{if } q \equiv 1 \mod 4
   \end{cases}\]

On the other hand, if $n$ is an  odd prime, then
\begin{equation}\label{eq:M2}
\M_\beta= \begin{cases}\frac{1}{2}\left(q^{n-1}-1\right),  & \text{if } \beta=0,\\
\frac{1}{2}\left(q^{n-1} +\eta(\beta)q^{\frac{n-1}{2}}\right), & \text{if } \beta \neq 0,
\end{cases}
\end{equation}
where $\eta$ denotes the quadratic character in $\F_q$.
    \end{proposition}
       \begin{proof}
     Observe that, by  \eqref{chartrace},
     \[2 \M_\beta = \frac{1}{q} \sum _{u \in \F_q}\bar\psi(u \beta)\X_u =: \frac{1}{q}\sum_{u \in \F_q} \bar\psi(u\beta)\sum_{\xi \in \Fn^*}\tilde\psi(u\xi^2),\]
     because each of $\pm \xi \in  \Fn^*$ yields the same square $\xi^2 \in \Fn^*$.  Hence
    \[2q\M_\beta=\sum_{u \in \F_q}\X_u= q^n-1 + \sum_{u\in\F_q^*} \bar \psi(u \beta)\X_u,  \]
    where  $\X_u=\sum_{\xi\in \Fn^*}\tilde \psi(u\xi^2)$.
    
    Suppose $n=2$  and $\beta \neq 0 \in \F_q$..   Then every $u \in \F_q^*$ is a square in $\F_{q^2}$ so that, by Lemma~\ref{lem:A}, ,  $\X_u=g_2(u) -1 =\varepsilon_1 q-1$.  Hence,
  
  \[2q\M_\beta= q^2-1+(\varepsilon_1 q-1) \sum_{u \in \F_q^*}\bar \psi(u\beta)=q^2-1-(\varepsilon_1 q-1), \]
   since $\sum_{u \in \F_q}\bar\psi(u\beta) =0$.
   Consequently, \eqref{eq:M1} holds
   
    Now, suppose, $n$ is an odd prime so that nonsquares in $\F_q$ remain nonsquare in $\F_{q^n}$. Thus, given a fixed nonsquare $c \in \F_q$, the elements of $\F_q^*$ can be written as a disjoint union $\{u^2: u \in \F_q^*\} \cup \{cu^2: u \in \F_q^*\}$, where each member of $\F_q^*$ appears twice. 
    
    This time, by Lemma~\ref{gauss}, if $ \beta =0$, then 
       \[2q\M_0= q^n-1 +\frac{1}{2}\sum_{u \in \F_q} (\X_{u^2}(\chi_1)+\X_{cu^2}(\chi_1) )=q^n-1-(q-1)= q^n-q  \]
    and this case of  \eqref{eq:M2} follows.  Finally suppose $\beta \neq 0$. We have
   \begin{align}
  2 q\M_{\beta} -(q^n-1)  = & \frac{1}{2}\left(\sum_{u\in \F_q^*} \bar\psi(\beta u^2)\X_{u^2}+\bar\psi(c\beta u^2)\X_{cu^2}\right) \nonumber\\
 = & \frac{1}{2}\left(\sum_{u\in \F_q^*} \bar\psi(\beta u^2)\X_1+\bar\psi(c\beta u^2)\X_c\right)  \label{psiX}  \\
= & \frac{1}{2}\{(\chi_2(\beta)\varepsilon_2 q^{1/2}-1)(\varepsilon_2 q^{n/2}-1)+ \nonumber \\ & (\chi_2(c\beta)\varepsilon_2 q^{1/2}-1)(\chi_2(c)\varepsilon_2 q^{n/2}-1)\} \nonumber \\
= & \chi_2(\beta)  q^{(n+1)/2}  +1, \nonumber
   \end{align}
 where,  at \eqref{psiX},  Lemma~\ref{lem:A} with $n=1$   and $\varepsilon_2$ given by \eqref{eps2} is applied to the sums  over $u \in \F_q^*$ in addition to the sums over
    $\xi \in \F_{q^n}^*$ and we note that $\varepsilon_2^2=1$.
   This yields \eqref{eq:M2} more generally.
       \end{proof}

\section {Mixed character sums}\label{sec:sums}
Assume throughout that $q$ is an odd prime and $n$ is prime with $(q,n)$ even  (although this is not essential for much of the discussion). 
We begin by introducing the notion of freeness. Let $m\mid q^n-1$. An element $\xi \in \F_{q^n}^*$ is \emph{$m$-free} if $\xi = \zeta^d$ for some $d\mid m$ and $\zeta\in\F_{q^n}^*$ implies $d=1$. It is clear that primitive elements are exactly those that are $q_0$-free, where $q_0$ is the square-free part of $q^n-1$. It is also evident that there is some relation between $m$-freeness and multiplicative order.
\begin{lemma}[\cite{huczynskamullenpanariothomson13}*{Proposition~5.3}]\label{lemma:m-free}
If $m\mid q^n-1$ then $\xi\in\F_{q^n}^*$ is $m$-free if and only if $\gcd\left( m,\frac{q^n-1}{\ord(\xi)} \right)=1$.
\end{lemma}
We introduced the quadratic character of $\Fn$  in Lemma~\ref{gauss}.     More generally, the multiplicative characters of $\Fn^*$  form a group isomorphic to $\Fn^*$.   Each character $\chi$  has order of degree $d$, a divisor of $q^n-1$, and $\chi(0)=0$, by defiinition.    In particular, we denote by $\chi_1$ the trivial multiplicative character  and by $\eta = \chi_2$ the quadratic character.

Vinogradov's formula yields an expression of the characteristic function of $m$-free elements in terms of multiplicative characters, namely:
\[
\omega_m(\xi) := \theta(m) \sum_{d\mid m} \frac{\mu(d)}{\phi(d)} \sum_{\ord(\chi_d) = d} \chi_d(\xi) ,
\]
where $\mu$ stands for the M\"{o}bius function and  $\theta(m) := \phi(m)/m$.  Also, here the inner sum suns through multiplicative characters $\chi_d$  of order $d$.

\begin{proposition}\label{propo:hybrid}  Let $q$ be an odd prime and  $\chi_d$ be a multiplicative character of order $d$, $u\in\F_q$.   Set
\[
A := \sum_{\xi\in\F_{q^n}} \chi_d(\xi) \tilde\psi(u\xi^2) .
\]
\begin{enumerate}
\item If $d=1$ and $u=0$, then $A = q^n-1$.
\item If $d=1$ and $u\neq 0$, then $|A| \leq  q^{n/2}+1$.
\item If $d\neq 1$ and $u=0$, then $A = 0$.
\item If $d\neq 1$ and $u\neq 0$, then $|A|\leq 2 q^{n/2}$.
\end{enumerate}
\end{proposition}
\begin{proof}
When $d=1$, since $\chi_1(0)=0$, we have that $A= g_n(u) -1$. From this, we immediately obtain the first and second items. The third item is a consequence of the orthogonality relations and the last item is implied by \cite{schmidt76}*{Theorem~2G}.
\end{proof}

Finally, the following is an improvement of the main result of \cite{katz89}, in the case $n=2$,
see \cite{cohen10}*{Lemma~3.3}.
\begin{lemma}\label{lem:katz}
Let $\theta\in\F_{q^2}$ be such that $\F_{q^2} = \F_q(\theta)$ and $\chi$ a non-trivial character. Set
\[ B := \sum_{\alpha\in\F_q} \chi (\theta+\alpha) . \]
\begin{enumerate}
\item If $\ord(\chi)\nmid q+1$, then $|B| = \sqrt{q}$.
\item If $\ord(\chi)\mid q+1$, then $B = -1$.
\end{enumerate}
\end{lemma}
\section{Conditions for even pairs}\label{sec:conditions}
Recall that $2$-primitive elements are exactly the squares of primitive elements. In other words, we are looking for a primitive element the trace of whose square is fixed to some $\beta\in\F_q$. With that in mind, following the analysis of Section~\ref{sec:sums}, we define the following
\[
\N_\beta(m) := \sum_{\xi\in\F_{q^n}} \omega_m(\xi) t_\beta (\xi^2) ,
\]
where $m\mid q^n-1$. In particular, our aim is to prove that  $\N_\beta(q_0)\neq 0$  (where we recall that $q_0$ stands for the square-free part of $q^n-1$) and note that, in fact, since $(q,n)$ is even, $\N_\beta(q_0)$ counts twice the number of $2$-primitive elements with trace $\beta$. Next, we compute:
\begin{align}\
\frac{\N_\beta(m)}{\theta(m)} & =  \frac 1q \sum_{\xi\in\F_{q^n}} \sum_{d \mid m} \frac{\mu(d)}{\phi(d)} \sum_{\ord(\chi_d)=d} \chi_d(\xi) \sum_{u\in\F_q} \bar\psi(u\beta) \tilde\psi(u\xi^2) \nonumber \\
 & = \frac 1q  \sum_{d \mid m} \frac{\mu(d)}{\phi(d)} \sum_{\ord(\chi)=d} \sum_{u\in\F_q}\bar\psi(u\beta) \X_u (\chi_d) , \label{eq:N1} 
\end{align}
where
\[ 
\X_u (\chi_d) :=  \sum_{\xi\in\F_{q^n}} \chi_d(\xi) \tilde\psi(u\xi^2) .
\] 

  Now observe that the terms on the right side of \eqref{eq:N1} corresponding to $d=1$ are given precisely by $2\M_\beta$ as evaluated in  Proposition~\ref{m=1} and note that this is where the discrepancy occurs in \cite{cohenkapetanakis20}, equation (4.5).
  
  We proceed to  consider the contribution of the terms  on the right side of \eqref{eq:N1} with $d>1$: call this quantity $R_\beta(m)$. The argument echoes that of the proof of Lemma~\ref{m=1} without the precision of the latter.
  
    By Proposition~\ref{propo:hybrid}(3) we can suppose $u\neq0$.    Observe that
  \begin{align*}
  \sum_{u \in \F_q^*}\bar\psi(u\beta)\X_u(\chi_d)& =\frac{1}{2}\left(\sum_{u \in \F_q^*}\bar\psi(u^2\beta)\X_{u^2}(\chi_d) +\sum_{u \in \F_q^*}\bar\psi(cu^2\beta)\X_{cu^2}(\chi_d)\right) \\
   &=\frac{1}{2}\left(\sum_{u \in \F_q^*} \bar\chi_d(u)\bar\psi(\beta u^2)\X_1(\chi_d) +\sum_{u \in \F_q^*}\bar \chi_d(u)\bar\psi(c\beta u^2)\X_c(\chi_d)\right) \\    
  \end{align*}

 Thus,  again  writing $\F_q^*$ as a disjoint union of squares and nonsquares (each counted twice), we have  
\begin{equation}\label{C}
R_\beta(m)= \frac{1}{2q}\sum_{ \substack{d|m \\ d>1}}\frac{\mu(d)}{\phi(d)}\sum_{\ord (\chi_d)=d}\sum_{u \in \F_q^*}\big(\bar{\psi}(u^2\beta)\bar\chi_d(u)\X_1(\chi_d)+\bar{\psi}(u^2c\beta)\bar\chi_d(u)\X_{c}(\chi_d)\big),
\end{equation}

To proceed  we distinguish between the cases   $\beta\neq 0$ and  $\beta=0$.
\subsection{The case  $\beta\neq 0$ and $n$ prime.}
In this situation (\ref{C}) can be rewritten as follows.
\begin{lemma}\label{D}
Assume $(q,n)$ is even and $\beta(\neq  0) \in \F_q$.  Then
 \[ 
 R_\beta(m)=\frac{1}{2q}\sum_{ \substack {d\mid m \\ d>1}}\frac{\mu(d)}{\phi(d)}\sum_{\ord (\chi_d)=d}\big(\overline{X_{\beta}(\chi_d)}\X_1(\chi_d)+\overline{X_{c\beta}(\chi_d)}\X_c(\chi_d)\big),
  \] 
 where $X_{\beta}(\chi_d)$ (with $\chi_d$ restricted to $\F_q$) is the sum $\sum_{u \in \F_q} \chi_d(u)\psi(u^2 \beta)$ (i.e., the sum $\X_{\beta}(\chi_d)$ over $\F_q$ rather than $\F_{q^n}$).
 \end{lemma}

Next, we present a lower bound for  ${\mathcal N}_\beta(m)$ which yields a condition for it to be positive.   A key point is that  in sums  over divisors $d$ of an integer $m$ with a factor $\mu(d)$ effectively  involve  only {\em square-free} divisors $d$ and we designate the number of such divisors by $W(m) = 2^{\nu(m)}$, where $\nu(m)$ is the number of distinct prime disisors of $m$,

\begin{theorem}\label{mainbound}
\label{H}
 Assume $(q,n)$ is even, where $q$ is an odd prime power and $n$ is a prime. Let  $\beta \in \F_q^*$ and  $m$ be an even divisor of $q^n-1$ with $m_Q$ be the product of those primes in $m$ which divide $Q= \frac{q^n-1}{q-1}$.  Then
  \begin{equation}\label{I}
  {\mathcal N}_\beta(m) \geq  \theta(m)q^{\frac{n-1}{2}} \left\{q^{\frac{n-1}{2}} - 4W(m)+2W(m_Q)+1\right\}.
  \end{equation}
  Hence $\N_\beta(q^n-1)$ is positive whenever
  \begin{equation}\label{eq:main}
  q^{\frac{n-1}{2}}>4W(q^n-1)-2W(Q)-1.
  \end{equation}
\end{theorem}
\begin{proof}From \eqref{eq:N1} we have 
\[\frac{\N_\beta(m)}{\theta(m)}=2M_\beta+R_\beta(m).\]
Moreover, by Propostion \ref{m=1},  $|2M_\beta - q^{n-1}| \leq q^{\frac{n-1}{2}}$
From Proposition~\ref{propo:hybrid}, for any $\beta \in \F_q^*$ and multiplicative character $\chi$, $|\X_\beta(\chi)| \leq 2 q^{\frac{n}{2}}$ and $|X_\beta(\chi)| \leq 2q^{\frac{1}{2}}$ .
 Indeed when $d\mid m_Q$, then $\chi_d$  restricted to  $\F_q$ is the trivial character and  $|X_\beta(\chi_d)| \leq q^{\frac{1}{2}}+1$.  In fact, we can be more precise about the latter bound (i.e., when $d|m_Q$). 
 For,  whether $n$ is even or odd, by \eqref{eps2},  $X_\beta(\chi_1)$ and $X_{c\beta}(\chi_1)$ take the two  real values $\pm q^{1/2} -1$ in either order. Thus, one of $|\big(\overline{X_{\beta}(\chi_d)}\X_1(\chi_d)|$  and $|\big(\overline{X_{c\beta}(\chi_d)}\X_c(\chi_d)|$ is bounded by $2(q^{1/2} -1)q^{n/2}$ and the other by  $2(q^{1/2} +1)q^{n/2}$. So their sum is bounded absolutely by $4q^{(n+1)/2}$  
 Thus, by  Lemma~\ref{D}, 
\[\frac{ \mathcal{ N}_\beta(m)}{\theta(m)q^{\frac{n-1}{2}}} \geq  q^{\frac{n-1}{2}} -1- 4(W(m)-W(m_Q)) -2(W(m_Q)-1).\]

 and the  result  follows.
   \end{proof}
%

%
%
%
\subsection{The case   $\beta=0$ and $n$ an odd prime.}
Next we suppose that $n$ is an odd prime  and $\beta=0$.
Now (\ref{C}) does not have a Gauss sum factor. We show that, to ensure that ${\mathcal N}_0(q^n-1)$ is positive, it suffices to show
that ${\mathcal N}_0(Q)$ is positive.
\begin{lemma}
\label{L}
Suppose $\xi \in \F_{q^n}$ is $Q$-free. 
Then there exists $c \in \F_q$ with $c\xi \in \F_{q^n}$ primitive. 
If, further, $\Tr(\xi^2) =0$, then $\Tr((c\xi)^2)=0$.
\end{lemma}
\begin{proof}
 It is possible that $q-1$ and $Q$ have a common prime factor (or factors), namely prime factors of $n$.  Express $q-1$ as a product $LM$, where $L$ and $M$ are coprime, such that $\xi$ is $QL$-free  and $\xi$ is an $m$-th power in $\F_{q^n}$ for each prime $m$ dividing $M$ (so $m\nmid QL$).  Hence, if $\gamma$ is a primitive element of $\F_{q^n}$, then $\xi=\gamma^{M_0t}$, where $t$ and $Q$ are coprime and $M_0$ is such that its square-free part is identical with the square-free part of $M$.  Define $g=\gamma^Q$, a primitive element of $\F_q$, and set $c=g^L=\gamma^{QL}$. Thus
 $c\xi=\gamma^{QLt+M_0t}$ is $QLM$-free, i.e.,  $(q^n-1)$-free.

 If actually $\Tr(\xi^2) =0$, then $\Tr((c\xi)^2)=\Tr(c^2\xi^2) =0$ since $c^2 \in \F_q$.
\end{proof}
After  Lemma~\ref{L},  to show there exists a $2$-primitive element with trace $0$, it suffices to show there exists a $Q$-free element $\xi$ siuch that $T(\xi^2)=0$.
\begin{lemma}
\label{lemma:C}
 Assume $(q,n)$ is even with $n$ an odd prime,  and that $m\mid Q$.
 \begin{equation}\label{M}
R_0(m)=\frac{q-1}{2q}\sum_{ \substack{d\mid m \\ d>1}}\frac{\mu(d)}{\phi(d)}\sum_{\ord (\chi_d)=d}(\X_1(\chi_d)+\X_{c}(\chi_d)).
\end{equation}
\end{lemma}
\begin{proof}
The above is an immediate consequence of \eqref{C}, after considering the fact that $\chi_d$ is trivial on $\F_q$ for every $d\mid Q$.
\end{proof}
\begin{theorem}\label{N}
Assume $(q,n)$ is even with $n$ an odd prime. Suppose that $m\mid Q$. Then, with $ \beta=0$,
 \begin{equation}\label{O}
   {\mathcal N}_0(m) \geq \theta(m)q^{\frac{n}{2}-1}\left\{q^{\frac{n}{2}}-2W(m)(q-1) \right\}.
\end{equation}
Consequently, if
\begin{equation} \label{O*}
q^{\frac{n}{2}}> 2W(Q)(q-1),
\end{equation}
 then ${\mathcal N}_0(q^n-1)>0$.
\end{theorem}
\begin{proof}
Lemma~\ref{lemma:C}, combined with Proposition~\ref{propo:hybrid} and Proposition~\ref{m=1}, yields 
\[   {\mathcal N}_0(m) \geq \theta(m)q^{\frac{n}{2}-1}\left\{q^{\frac{n}{2}}-q^{1-n/2}-2(W(m)-1)(q-1) \right\}.\]
and \eqref{O} easily follows.

 Now, assume that \eqref{O*} holds. By \eqref{O}, there exists some $Q$-free $\zeta\in\F_{q^n}$ with $\Tr(\zeta^2)=0$ and from Lemma~\ref{L} this implies the existence of a primitive $\xi\in\F_{q^n}$ such that $\Tr(\xi^2)=0$.
\end{proof}
%
%
%
%
\section{Sieving conditions} \label{sec:sieve}
Our next aim is to modify  the conditions  in Section~\ref{sec:conditions} by adopting a  well-established prime  sieving technique (see \cite{cohenhuczynska03}).  For any divisor $m$ of $q^n-1$ in expressions such as ${\mathcal N}_\beta(m)$ we freely interchange $m$ and its radical, i.e., the product of distinct primes dividing $m$.
\begin{proposition}[Sieving inequality]\label{prop:siev0}
Assume $(q,n)$ is even. Let $m\mid q_0$ (the square-free part of $q^n-1$) and $\beta\in\F_q$. Write $m=kp_1\ldots p_s$, where $p_1, \ldots, p_s$  are distinct prime divisors of $m$.  Then
\[
\N_\beta(m) \geq \sum_{i=1}^s \N_\beta(kp_i) - (s-1)\N_\beta(k) .
\]
\end{proposition}


%
First suppose $\beta \neq 0$ and $n$ is a  prime.  Let the radical   of $q^n-1$  be expressed as $kp_1\ldots p_s$, where $p_1,\ldots,p_s$ are distinct primes and
 $s \geq 0$ and define $\delta=1 - \sum_{i=1}^s \frac{1}{p_i}$, with $\delta =1$, if $s=0$. Suppose further that $p_i\mid Q$ for $i=1, \ldots,r$ and $p_i\nmid Q$ for $i=r+1,\ldots,s$.  Set $\delta_Q= 1- \sum_{i=1}^r\frac{1}{p_i}$.
\begin{theorem}
\label{T}
Assume $(q,n)$ is even with $n$ a  prime.  Suppose $\beta\neq 0$. Define $\delta, \delta_Q$ as above and assume that $\delta$ is positive. Then
\begin{multline}
\label{S}
{\mathcal N}_\beta(q^n-1)\geq
 \delta\theta(k)q^{\frac{n-1}{2}}\bigg\{ q^{\frac{n-1}{2}}  \\ -4\left(\frac{s-1}{\delta}+2\right)W(k) + 2\left(\frac{r-1+\delta_Q}{\delta}+1\right)W(k_Q)\bigg\}.
\end{multline}
Hence, if
\[
 q^{\frac{n-1}{2}}> \\ 4\left(\frac{s-1}{\delta}+2\right)W(k)-2\left(\frac{r-1+\delta_Q}{\delta}+1\right)W(k_Q),\]
then ${\mathcal N}_\beta(q^n-1) >0$.
\end{theorem}
\begin{proof}
Proposition~\ref{prop:siev0} implies that, for any $\beta \in  \F_q$,
\begin{align}\label{P}
{\mathcal N}_\beta(q^n-1) \geq & \sum_{i=1}^s {\mathcal N}_\beta(kp_i) -(s-1){\mathcal N}_\beta(k)\nonumber\\
\geq & \delta{\mathcal N}_\beta(k) -\sum_{i=1}^s \left|{\mathcal N}_\beta(kp_i)-\left (1- \frac{1}{p_i}\right){\mathcal N}_\beta(k)\right|
\end{align}

In (\ref{P}) use (\ref{I}) with $m=k$ as a lower bound. For the absolute value of the difference expressions  we distinguish between values two cases according as  $p_i\mid Q$ or not.
Suppose $p_i\nmid Q$.  Then
\begin{equation}\label{Q}
\left|{\mathcal N}_\beta(kp_i)-\left (1- \frac{1}{p_i}\right){\mathcal N}_\beta(k) \right|\leq 4\theta(k)\left(1-\frac{1}{p_i}\right) q^{\frac{n-1}{2}}W(k),
\end{equation}
since $W(kp_i)-W(k)=W(k)$.  On the other hand, if $p_i\mid Q$, we have the improved bound
\begin{equation}\label{R}
\left|{\mathcal N}_\beta(kp_i)-\left (1- \frac{1}{p_i}\right){\mathcal N}_\beta(k) \right|\leq  \theta(k)\left(1-\frac{1}{p_i}\right) q^{\frac{n-1}{2}}\left\{4W(k)-2W(k_Q)\right\},
\end{equation}
 using also the fact that $W(k_Qp_i)-W(k_Q)=W(k_Q)$.
By combining (\ref{I}), (\ref{P}), (\ref{Q}) and (\ref{R})  we deduce that  (\ref{S}) holds.
\end{proof}
Finally, suppose $\beta=0$ and $n$ is an odd prime. We use the sieve version of the criterion (\ref{O*}) to obtain a result that depends on writing $Q$ (rather than $q^n-1$) as $Q= kp_1\ldots p_s$.
\begin{theorem}
\label{Z}
Assume $(q,n)$ is even with $n\geq 3$.
  With the notation  $Q= kp_1\ldots p_s$, with $p_1, \ldots, p_s$ distinct primes dividing $Q$, set $\delta = 1- \sum_{i=1}^s\frac{1}{p_i}$. Assume that $\delta$ is positive.  Then
\[
{\mathcal N}_0(Q)>
 \delta\theta(k)q^{\frac{n}{2}}\left\{q^{\frac{n}{2}-1}-2\left(\frac{s-1}{\delta}+2\right)W(k)\right\}.
\]
Hence, if
\[
 q^{\frac{n}{2}-1}>2\left(\frac{s-1}{\delta}+2\right)W(k),
\]
then ${\mathcal N}_0(q^n-1) >0$.
\end{theorem}
\begin{proof}
The proof follows the same pattern as that of Theorem~\ref{T}, this time with the difference being that \eqref{O} substitutes for \eqref{I}.
\end{proof}
For a multiplicative character $\chi$ of
  $\F_{q^n}$ denote by $G_n(\chi)$ the Gauss sum $G_n(\chi)=\sum_{\xi \in \F_{q^n}}\chi(\xi)\psi(\xi)$, where $\psi$ is the canonical additive character.  In particular, $G_n(\chi_2)= g_n(1)$ as used in Lemma~~\ref{gauss}.    Indeed, by Lemma~\ref{gauss} we have

\[ \sum_{\xi\in\F_{q^n}} \psi (b\xi^2) = \chi_2(b) G_n(\chi_2) . \]

In the case in which $q$ is prime and $n=1$, the following lemma is established in \cite{wenpeng02}*{Lemma~4}. Here, we prove it more generally.
\begin{lemma}
\label{A1}
Assume $(q,n)$ is even. Let $\chi$ be any non-trivial multiplicative character of $\F_{q^n}$. Then, for any $b\in \F_{q^n}$,
\begin{equation} \label{B1}
|\X_b(\chi)|^2 = (1+\chi(-1))q^n+\chi_2(b)G_n(\chi_2)C(\chi),
\end{equation}
where $C(\chi):= \sum_{\xi \in  \F_{q^n}} \chi(\xi)\chi_2(\xi^2-1)$. Thus $|C(\chi)|\leq 2q^{\frac{n}{2}}$.
\end{lemma}
\begin{proof}
Let $\psi$ be the canonical additive character of $\F_{q^n}$. We have that
\begin{align*}
|\X_b(\chi)|^2 & = \sum_{\xi\in\F_{q^n}^*} \chi(\xi)\psi(b\xi^2) \overline{\sum_{\zeta\in\F_{q^n}^*} \chi(\zeta)\psi(b\zeta^2)} \\
 & = \sum_{\xi\in\F_{q^n}^*}\sum_{\zeta\in\F_{q^n}^*} \chi \left( \frac{\xi}{\zeta} \right) \psi(b(\xi^2 - \zeta^2)) \\
 & = \sum_{\xi\in\F_{q^n}^*} \sum_{\zeta\in\F_{q^n}^*} \chi(\xi) \psi (b\zeta^2(\xi^2-1)) \\
 & = \sum_{\xi\in\F_{q^n}^*} \chi(\xi) \left[ \sum_{\zeta\in\F_{q^n}} \psi (b\zeta^2(\xi^2-1)) -1 \right] \\
 & = (1+\chi(-1)) q^n + \sum_{\substack{\xi\in\F_{q^n}^* \\ \xi\neq\pm 1}} \chi(\xi) \sum_{\zeta\in\F_{q^n}} \psi (b\zeta^2(\xi^2-1)) - \sum_{\xi\in\F_{q^n}^*} \chi(\xi) .
\end{align*}
The result now follows from Lemma~\ref{gauss}.
\end{proof}
As we know from Lemma~\ref{lem:A}, when $(q,n)$ is even, we have $G_n(\chi_2) = \pm q^{\frac{n}{2}}$.  We proceed with the implications of Lemma~\ref{A1} when $n$ is odd; in particular it applies in the key case when $n=3$. In this situation, since $(q,n)$ is even, necessarily $q \equiv 1 \pmod 4$.

\begin{lemma}
\label{C1}
Assume $q \equiv 1 \pmod 4$ and $n$ is an  odd prime.  Let $\chi$ be a non-trivial multiplicative character of $\F_{q^n}$ and $c$ a nonsquare in $\F_q$.  Then
\[
|\X_1(\chi)| +|\X_c(\chi)| \leq 2 \sqrt{2} q^{\frac{n}{2}}.
\]
\end{lemma}
\begin{proof}
Since in this context $Q$ is odd then $c$ remains a nonsquare in $\F_{q^n}$.  Thus $\chi_2(c)=-1=-\chi_2(1)$.  Hence, from (\ref{B1}),
\begin{equation}\label{eq:ip_C1}
\big(|\X_1(\chi)| +|\X_c(\chi)|\big)^2 =2(1+\chi(-1))q^n + 2|\X_1(\chi)||\X_c(\chi)| .
\end{equation}
Additionally, in a similar manner as in the proof of Lemma~\ref{A1}, we have that
\begin{align*}
\X_1(\chi)\overline{\X_c(\chi)} & = \sum_{\xi\in\F_{q^n}^*} \sum_{\zeta\in\F_{q^n}^*} \chi \left( \frac\xi\zeta \right) \psi(\xi^2 - c\zeta^2) \\
 & = \sum_{\xi\in\F_{q^n}^*} \sum_{\zeta\in\F_{q^n}^*} \chi(\xi) \psi(\zeta^2(\xi^2-c)) \\
 & = \sum_{\xi\in\F_{q^n}^*} \chi(\xi) \left[ \sum_{\zeta\in\F_{q^n}} \psi(\zeta^2(\xi^2-c)) - 1 \right] \\
 & = \sum_{\xi\in\F_{q^n}^*} \chi(\xi) \sum_{\zeta\in\F_{q^n}} \psi(\zeta^2(\xi^2-c)) .
\end{align*}
Now Lemma~\ref{gauss} yields that
\[ \X_1(\chi)\overline{\X_c(\chi)} = G_n(\chi_2) \sum_{\xi\in\F_{q^n}^*} \chi(\xi) \chi_2(\xi^2-c) . \]
From the fact that $G_n(\chi_2) = \pm q^{n/2}$ and that the (absolute value of the) inner sum is bounded by $2q^{n/2}$, it follows that
\[ |\X_1(\chi)||\X_c(\chi)| =  |\X_1(\chi)\overline{\X_c(\chi)}| \leq 2q^n \]
and the result follows once we insert the above in \eqref{eq:ip_C1}.
\end{proof}
By applying Lemma~\ref{C1} to (\ref{M}) (instead of  $|\X_b(\chi)| \leq 2q^{\frac{3}{2}}$)  and extending this to the sieve result we obtain the following improvements to Theorems~\ref{T} and \ref{Z}.  
\begin{theorem}
\label{E1}
Assume $q\equiv 1 \pmod 4$ and $n$ is  an odd prime.  With  the notation of Theorem~$\ref{T}$, assume  $\beta \neq 0$ and $\delta >0$.  Suppose
\[
  q^{\frac{n-1}{2}} >2\sqrt{2}\left(\frac{s-1}{\delta}+2\right)W(k)-\sqrt{2}\left(\frac{r-1+\delta_Q}{\delta}+1\right)W(k_Q).
\]
Then $\mathcal{N}_\beta(q^n-1) >0$.
\end{theorem}
\begin{theorem}
\label{G1}
Assume $q\equiv 1 \pmod 4$ and $n$ is an  odd prime.  In the situation insofar of  Theorem~$\ref{Z}$, assume  $\beta =0$ and $\delta >0$. Suppose that
\[
q^{\frac{n}{2}-1}>\sqrt{2}\left(\frac{s-1}{\delta}+2\right)W(k).
\]
Then $\mathcal{N}_0(q^n-1) >0$.
\end{theorem}

%
%
\section{Extensions of odd prime degree}\label{sec:existence}
In this section we complete the proof of Theorem~\ref{thm:main} for $n$ an odd prime, insofar as it can be accomplished  theoretically.
 By Lemma~\ref{induction} we can assume $n$ is prime.   If  $q$ is odd, then $q \equiv 1  \pmod 4$.   We distinguish the cases, $n$ a prime exceeding $4$ and  $n=3$, while the case $n=2$ is studied in Section~\ref{sec:quadratics}.
Recall $W(t) =2^{\nu(t)}$ is  the number of the square-free divisors of $t$. The following provides a bound for this number.
\begin{lemma}\label{lem:w(r)}
Let $t, \ell$ be positive integers and let $p_1, \ldots, p_j$ be the distinct prime divisors of $t$ such that $p_i\le 2^\ell$. Then $W(t)\le c_\ell(t)t^{1/\ell}$, where
\[ c_\ell(t)=\frac{2^j}{(p_1\cdots p_j)^{1/\ell}}. \]
In particular,
$c_4(t)< 4.87$ for every $t$.  Indeed, if $t$ is odd, then $c_4(t) < 2.9$. 

Further, for any $t$, $c_6(t)<  46.103$.
 
\end{lemma}
\begin{proof}
The statement is an immediate generalization of \cite{cohenhuczynska03}*{Lemma~3.3} and can be proved using multiplicativity.
\end{proof}

Other specific  applications of Lemma~\ref{lem:w(r)} will be given where they are used. We comment that, in \cite{cohenkapetanakis20}, less successfully, we used a bound for $c_8(t)$.

\subsection{The case  $n>4$, prime}
We suppose $n \geq 5$ is prime and that $4|q-1$ (so that $q \ge 5$). Take any $\beta \in \F_q^*$.   We begin by employing the simplest condition for $\N_\beta(q_0)\neq 0$ to check, that is
\begin{equation}\label{eq:n>4_1}
q^{\frac{n}{4} -\frac{1}{2}} >4 \cdot 4.87/2^{\frac{1}{4}}=16.38 \ldots   ,
\end{equation}
which is a consequence of \eqref{eq:main} and Lemma~\ref{lem:w(r)}.   Now, \eqref{eq:n>4_1} is satisfied for $n\geq 17$ and $q\geq 3$, which means that the case $n>13$ is settled.  It also holds if $q>3$ when $n=11$, if $q>7$  when $n=7$ and if $q>41$ when $n=5$.  The remaining cases can be checked using the simple sufficient condition $q^{\frac{n-1}{2}}>4W(q^5-1)$ except when $(q,n)=(5,5)$. In this last case $5^5-1 =2^2\cdot11 \cdot 71$ so that $Q=11\cdot 71$.  .  With $m=5^5-1$ in Theorem~\ref{mainbound}, $\N_\beta(5^5-1)$ is positive
since $q^2=25 >4W(q^n-1) -2W(Q)-1=32-8-1=23$.

Now, take  $\beta=0$.    Since $Q$ is odd,  after Theorem~\ref{N} with Lemma~\ref{L} and Lemma~\ref{lem:w(r)}, it suffices to show that
\begin{equation}\label{eq:Qodd}
\frac{q^{\frac{n}{2}}}{q-1}>2\cdot 2.9 \cdot \left(\frac{q^n-1}{q-1} \right)^{\frac{1}{4}}, 
\end{equation}
which can be written
\[ \frac{x^2}{x-1} >5.8^4(q-1)^3,\]
where $x=q^n$.  Now the function $x^2/(x-1) $ is increasing for $x>2$ and so, if \eqref{eq:Qodd} holds for the pair $(q,n)$ it holds, then it holds for $(q, n_1)$ where $n_1>n$. Moreover, \eqref{eq:Qodd} holds when $n=7$ and $q \geq 5$.  Hence it holds for larger primes than $7$.    It also holds for $n=5$ provided $q>32$.  Further, if $ n=5$ and $q<32$, we can suppose $Q$ is  such   that its  prime divisors less than 16 are (at most) 5 and 11.  Hence, with $t=Q$ in Lemma~\ref{lem:w(r)}, $c_4(Q)<1.469$ and we can replace 2.9 by 1.469 on the right side of \eqref{eq:n>4_1}  with $n=5$ and this  satisfied whenever $n>6$.  Finally, when $(q,n)=(5,5)$ we have $c_4(Q)<1.099$ and \eqref{eq:Qodd} is satisfied with $2.9$ replaced by $1.099$.

Summarising we have established the following.

\begin{proposition}
\label{prop:main_n>4}
Let $q$ be an odd prime power and $n>4$ a prime with $q \equiv 1 \mod 4$. Then, for any $\beta\in\F_q$, there exists a $2$-primitive  $\xi\in\F_{q^n}$ with 
$\Tr(x)=\beta$.   
\end{proposition}

\subsection{The case  $n=3$}
Here we assume that  $n=3$ and $\beta\in\F_q$, where $q \equiv 1 \pmod {4}$ is prime  (so that $ 4|q^3-1$ and $q \ge 5$))

First suppose  $\beta \neq 0$.  Then, with $n=3$,  using  \eqref{eq:n>4_1}  and Theorem~\ref{E1} we obtain the sufficient condition
\[ q^{\frac{1}{4}}  > 2\sqrt{2}\cdot4.87/ 2^{\frac{1}{4}} =11.5828\ldots  . \]
 This is satisfield if  $q>18000$.   So assume $q <18000$ so that $q^3-1<5.5833 \cdot 10^{12}$.   This implies that $\nu(q_0) \leq 11$.   Assuming for the moment that also $\nu(q_0)  \geq 9$,  apply the sieving condition Theorem~\ref{E1} with  $ \nu(k)=2$  (relating to the two smallest primes dividing $q^3-1$) and $s\leq 9$.  The largest  $\delta$ would occur if the set of sieving primes comprise those  from 5 to  31.  Thus $\delta >0.26763$ and the condition is satisfied if $q \geq 361$.  So we can assume $q<361$ and $q^3-1<4.7046\cdot 10^7$. This implies  $\nu (q_0) \leq 8$.  Another round of the sieve yields  a condition that is saisfied  if $q>173$.     So we can assume $q \leq 1169$.   A final general round of sieving produces a condition that is satisfied if $q >128$ so we can assume $q \leq 125$.
 
 The next stage is to apply the full condition of Theorem~\ref{E1} to the remaining prime powers $\le 125$  using the precise decomposition.   Even without sieving (i.e., with $k=q_0$) this yields a condition that is satisfied unless  $q=5$, $9$, $13$, $25$, $29$, $61$ and $121$.  We successfully apply sieving for $q=29$, $61$ and $121$, with $\{67,13,7\}$, $\{ 97,13,5 \}$ and $\{ 37,19,7 \}$ as our set of sieving primes respectively. This concludes the case $\beta\neq 0$.

Now suppose $\beta=0$.  From Theorem~\ref{G1} with $k=Q$, the basic condition to be satisfied is
\[ q^{\frac{1}{2}} > \sqrt{2}W(Q).\]
Now the primes dividing $Q=q^2+q+1$ can only be $3$ or primes $\equiv 1 \mod 6$. Thus, by Lemma~\ref {lem:w(r)} with $\ell=6$ and  $t=Q$ (so the possible  primes less than 64  dividing $t$ lie in the set $\{3,7,13,19,31,37, 43,61\}$ we have $c_6(Q)<5.1211$ and a sufficiernt condition  is
 \[ q^{\frac{1}{2}} >\sqrt{2} \cdot 5.1211\cdot (q^2+q+1)^{\frac{1}{6}}\]
 and this is satisfied when $q \ge 144303$. So assume $q <144303$ which means that $Q< 2.08236 \cdot 10 ^{10}$. Taking into account the nature of the possible primes dividing $Q$, this implies $\nu(Q)\leq 8$.    Temporarily assuming additionally that $\nu(Q) \ge 6$,  apply Theorem~\ref{G1} with $\nu(k)=1$ and $s\leq 7$ so that $\delta\ge 1-\frac{1}{7}-\frac{1}{13} - \cdots -\frac{1}{61} >0.62865$.      Then the condition is satisfied if  $q>1067$.   So, assume $q<1067$ and apply another round of sieving.   We have $\nu(Q) \le 5$ and take $k=3$ and $ \delta >0.69633$. The condition then holds if $q>319$.  A final round of general sieving leads to a condition that is satisfied if $q>185$.  Hence we may assume $q \leq 181$.
By applying  Theorem~\ref{G1} without sieving  using exact prime decompositions yields a condition that is satisfied except when $q=5,9,13,25,29, 37, 49, 61, 81, 109,121$ (1121 values).   Four of these, however, succumb to the sieving process as follows: $q=29$ with sieving primes $\{13,67\}$; $61, \{3,13,97\}$; $81, \{13,73\}$; $109, \{7,571\}$.

Summing up, we have proved the following.
\begin{proposition}
\label{prop:main_n=3}
Let $q$ be a prime power such that $q\equiv 1\pmod{4}$. For any $\beta\in\F_q$, there exists a $2$-primitive $\xi\in\F_{q^3}$  with $\Tr(\xi)=\beta$, unless $q=5$, $9$, $13$ or $25$ and $\beta\in\F_q$ or $q=37$, $49$ or $121$ and $\beta=0$.
\end{proposition}
\section{Quadratic extensions}\label{sec:quadratics}
For $n=2$ we recall that we may assume that $\beta\neq 0$.   From Theorem~\ref{mainbound}, we have the sufficient criterion that $\N_\beta(q^2-1)$ is positive whenever
\[ q^{1/2} > 4W(q^2-1)-2W(q+1)-1.\]
Similarly, there is a corresponding sieving criterion derivable from Theorem~\ref{T}.  The adoption, however, of a  strategy found  in \cite{cohen90} yields stronger results, so we repeat  this approach here.
\begin{lemma}\label{lemma:basis}
For every $\beta\in\F_q^*$, there exist $\theta_1,\theta_2\in\F_{q^2}$, such that $\{\theta_1,\theta_2\}$ is an $\F_q$-basis of $\F_{q^2}$, $\Tr(\theta_1)=\beta$ and $\Tr(\theta_2)=0$.
\end{lemma}
\begin{proof}
The trace function is onto, hence there exists some $\theta_1\in\F_{q^2}$ such that $\Tr(\theta_1)=\beta$. Next, extend $\{\theta_1\}$ to an $\F_q$-basis of $\F_{q^2}$, say $\{\theta_1,\theta_2'\}$ and set $\theta_2 := \theta_2' - \frac{\Tr(\theta_2')}{\Tr(\theta_1)} \cdot \theta_1$. It is clear that $\{\theta_1,\theta_2\}$ satisfies the desired conditions.
\end{proof}
\begin{corollary}\label{cor:basis}
Let $\beta,\theta_1,\theta_2$ be as in Lemma~$\ref{lemma:basis}$. For every $\alpha\in\F_q$, we have that $\Tr(\theta_1+\alpha\theta_2)=\beta$.
\end{corollary}
Fix $\beta\in\F_q^*$ and let $\theta_1,\theta_2$ be as in Lemma~\ref{lemma:basis}. Since $\{\theta_1,\theta_2\}$ are $\F_q$-linearly independent, we have that $\theta_1/\theta_2\not\in\F_q$, that is $\F_{q^2}=\F_q(\theta_1/\theta_2)$. In addition, Corollary~\ref{cor:basis} implies that for every $\alpha\in\F_q$, $\Tr(\theta_1+\alpha\theta_2)=\beta$.

Write $q^2-1=2^\ell q_2$, where $q_2$ is odd, and notice that, since $q$ is odd, $8\mid q^2-1$, that is $\ell \geq 3$, while the fact that $\gcd(q-1,q+1)=2$ implies that $q_2=r_2s_2$ where $r_2$ and $s_2$ are the $2$-free parts of $q+1$ and $q-1$ respectively and they are co-prime. Also, set $q_2',r_2'$ and $s_2'$ as the square-free parts of $q_2,r_2$ and $s_2$ respectively.

Next, take $r\mid q_2'$ and set $\Q_r$ to be the number of $r$-free elements of the form $\theta_1+\alpha\theta_2$ for some $\alpha\in\F_q$, that are squares but not 4th powers. Following the analysis of Section~\ref{sec:sums}, we get that
\begin{align}
\Q_r  & = \sum_{x\in\F_q} \omega_{r}(\theta_1+x\theta_2) w_2(\theta_1+x\theta_2) (1-w_4(\theta_1+x\theta_2)) \nonumber \\
 & =  \sum_{x\in\F_q} \omega_{r}(\theta_1+x\theta_2) (w_2(\theta_1+x\theta_2) -w_4(\theta_1+x\theta_2)) , \label{eq:Qr}
\end{align}
given that, by definition, for all $\xi\in\F_{q^2}^*$, $w_2(\xi)w_4(\xi) = w_4(\xi)$. In addition, notice that, for all $\xi\in\F_{q^2}^*$, we have that
\begin{align}
w_2(\xi)-w_4(\xi) & = \frac{1}{2} \sum_{\delta\mid 2} \sum_{\ord(\chi_\delta)=\delta} \chi_\delta(\xi) - \frac{1}{4} \sum_{\delta\mid 4} \sum_{\ord(\chi_\delta)=\delta} \chi_\delta(\xi) \nonumber \\
 & = \frac{1}{2} \sum_{\delta\mid 4} \sum_{\ord(\chi_\delta) = \delta} \ell_{\delta} \chi_\delta(\xi) , \label{eq:w2-w4}
\end{align}
where,
\[ 
\ell_{\delta} := \begin{cases} 1/2 , & \text{if } \delta=1\text{ or }2 , \\ -1/2 , & \text{if } \delta=4 . \end{cases}
\]
Furthermore, Lemma~\ref{lemma:m-free} implies that an element is $q_2'$-free if and only if it is $2^i$-primitive for some $0\leq i\leq\ell$. It follows that $2$-primitive elements of $\F_{q^2}$ are the $q_2'$-free elements that are squares, but not 4th powers. In other words, it suffices to show that $\Q_{q_2'}\neq 0$, while it is clear that 
\[\Q_{q_2'}\neq 0 \Rightarrow \N_\beta(q^2-1) \neq 0 . \]
 
In \eqref{eq:Qr}, we replace $\omega_{r}$ by its expression and $w_2-w_4$ by its expression in \eqref{eq:w2-w4} and we get that
\begin{equation}\label{eq:Q1}
\frac{4\Q_r}{\theta(r)} = \left( \sum_{\substack{d\mid r \\ \delta\mid 4}} \frac{\mu(d)}{\phi(d)} 2 \ell_\delta \sum_{\substack{\ord(\chi_d)=d \\ \ord(\chi_\delta)= \delta}} \Y(\chi_d,\chi_\delta) \right)  = \left( \sum_{d\mid r}\frac{\mu(d)}{\phi(d)} \sum_{\ord(\chi_d)=d} \Z(\chi_d) \right),
\end{equation}
where
\[
\Y(\chi_d,\chi_\delta) := \sum_{\alpha\in\F_q} \psi_{d,\delta} (\theta_1+\alpha\theta_2) = \psi_{d,\delta}(\theta_2)\sum_{\alpha\in\F_q} \psi_{d,\delta} \left( \frac{\theta_1}{\theta_2} + \alpha \right)
\]
and
\[
\Z(\chi_d) := \Y(\chi_d,\chi_1) + \Y(\chi_d,\chi_2) - \Y(\chi_d,\eta_1) -\Y(\chi_d,\eta_2) ,
\]
where $\psi_{d,\delta}:=(\chi_d\chi_\delta)$ is the product of the corresponding characters, $\eta$ is the quadratic character and $\eta_1,\eta_2$ are the two multiplicative characters of order exactly $4$. Furthermore, since $d$ is odd and $\delta\mid 4$, it is clear that $\psi_{d,\delta}$ is trivial if and only if $d=\delta=1$. 

Recall that $\F_{q^2} = \F_q(\theta_1/\theta_1)$.
First, assume $q\equiv 1\pmod{4}$. Then $4\nmid q+1$. Hence Lemma~\ref{lem:katz} implies that
\begin{enumerate}
  \item for $\chi_1$, $|\Z(\chi_1)| \geq q-1-2\sqrt{q}$,
  \item for $1\neq \ord(\chi_d)\mid q+1$, $|\Z(\chi_d)| \leq 2+2\sqrt{q}$, 
  \item for $\ord(\chi_d)\nmid q+1$, $|\Z(\chi_d)| \leq 4\sqrt{q}$.
\end{enumerate}
Next,  assume that $q\equiv 3\pmod{4}$.  Then $4\mid q+1$ and Lemma~\ref{lem:katz} implies that
\begin{enumerate}
  \item for $\chi_1$, $|\Z(\chi_1)|\geq q-3$,
  \item for $1\neq \ord(\chi_d)\mid q+1$, $|\Z(\chi_d)| \leq 4$,
  \item for $\ord(\chi_d)\nmid q+1$, $|\Z(\chi_d)| \leq 4\sqrt{q}$.
\end{enumerate}
We substitute the above in \eqref{eq:Qr} and arrive at  the following.
\begin{proposition}\label{prop:n=2(1)}
Let $q$, and $r$ be as above and let $r_1$ be the product of the prime divisors of $r$ that divide $q+1$.
\begin{enumerate}
\item If $q\equiv 1\pmod{4}$, then
\begin{equation} \label{eq:n=2_ip(1)}
\frac{4 \Q_r}{\theta(r)} \geq q+1 - 4W(r)\sqrt{q} + 2W(r_1) (\sqrt{q}-1);
\end{equation}
that is, if
\[ q+1 > 4\left( W(r)\sqrt{q} - W(r_1) \left( \frac{\sqrt{q}-1}{2} \right)\right) ,\]
then $\Q_r\neq 0$.
\item If $q\equiv 3\pmod{4}$, then
\begin{equation} \label{eq:n=2_ip(2)}
\frac{4\Q_r}{\theta(r)} \geq q+1 - 4W(r) \sqrt{q} + 4W(r_1)(\sqrt{q}-1);
\end{equation}
that is, if
\[ q+1 > 4(W(r)\sqrt{q} - W(r_1)(\sqrt{q}-1)) ,\]
then $\Q_r\neq 0$.
\end{enumerate} 
\end{proposition}

Again, as in previous sections, we employ a sieving inequality as follows.
\begin{proposition}[Sieving inequality]\label{prop:siev1}
Let $r\mid q_2'$. Write   $r=kp_1\cdots p_s$ , where $p_1, \ldots, p_s$ are distinct prime divisors of $r$. Then
\[
\Q_{r} \geq \sum_{i=1}^s \Q_{r_i} - (s-1)\Q_{r_0} .
\]
\end{proposition}

Write $q_2'=kp_1\cdots p_s$, where $p_1$,\ldots,$p_s$ are distinct primes and $\varepsilon := 1-\sum_{i=1}^s 1/p_i$, with $\varepsilon=1$ when $s=0$. Further, suppose that $p_i\mid q+1$ for $i=1,\ldots,r$ and $p_i\nmid q+1$ for $i=r+1,\ldots,s$. Finally, set $\varepsilon' := 1-\sum_{i=1}^r 1/p_i$ and let $k_1$ be the part of $k$, that divides $q+1$.
\begin{theorem}\label{thm:siev4}
Let $q$ and $q_2'$ be as above. Additionally, let $\varepsilon$ and $\varepsilon'$ be as above and assume that $\varepsilon>0$.
\begin{enumerate}
\item If $q\equiv 1\pmod{4}$ and
\[ q+1 > 4\left[ W(k)\left( \frac{s-1}{\varepsilon} +2 \right) \sqrt{q} - W(k_1) \left( \frac{r-1+\varepsilon'}{\varepsilon} +1 \right) \left( \frac{\sqrt{q}-1}{2} \right)\right] ,\]
then $\Q_{q_2'}\neq 0$.
\item If $q\equiv 3\pmod{4}$ and
\[ q+1 > 4\left[  W(k)\left( \frac{s-1}{\varepsilon} +2 \right)\sqrt{q}- W(k_1) \left( \frac{r-1+\varepsilon'}{\varepsilon} +1 \right) (\sqrt{q}-1) \right]  ,\]
then $\Q_{q_2'}\neq 0$.
\end{enumerate} 

In particular, it is the case that $\Q_{q_2'} > 0$ whenever
\begin{equation}\label{eq:both}
q\ge4 W(k) \left(\frac{s-1}{\varepsilon}+2 \right).
\end{equation}
\end{theorem}
\begin{proof}
Proposition~\ref{prop:siev1} implies that
\[
\Q_{q_2'}\ \geq\ \sum_{i=1}^s \Q_{kp_i} - (s-1) \Q_{k} 
\ \geq\ \varepsilon \Q_{k} - \sum_{i=1}^s \left| \Q_{kp_i} - \left( 1-\frac{1}{p_i} \right) \Q_{k} \right| . \label{eq:siev_ip(0)}
\]
Notice that $\theta(kp_i) = \theta(k)(1-1/p_i)$. It follows from \eqref{eq:Q1} that
\begin{equation} \label{eq:siev_ip(1)}
\Q_{kp_i} - \left( 1-\frac{1}{p_i} \right) \Q_{k} =\frac{\theta(k)(p_i-1)}{4p_i} \sum_{d\mid k} \frac{\mu(dp_i)}{\phi(dp_i)} \sum_{\ord(\chi_{dp_i})=dp_i} \Z(\chi_{dp_i}) .
\end{equation}

First assume that $q\equiv 1\pmod{4}$. We repeat the arguments that led us to \eqref{eq:n=2_ip(1)} for \eqref{eq:siev_ip(1)}. If $i=1,\ldots,r$, i.e., $p_i\mid q+1$, then
\begin{multline*}
\left| \Q_{kp_i} - \left( 1-\frac{1}{p_i} \right) \Q_{k} \right| \leq \\ \theta(k) \left( 1-\frac{1}{p_i} \right) \Big[ 2\sqrt{q} (W(k)-W(k_1)) + (1+\sqrt{q}) W(k_1) \Big] ,
\end{multline*}
since $W(kp_i) = 2W(k)$ and $W(k_1p_i) = 2W(k_1)$. Similarly, if $i=r+1,\ldots,s$, i.e., $p_i\nmid q+1$, then
\[
\left| \Q_{kp_i} - \left( 1-\frac{1}{p_i} \right) \Q_{k}(\theta,\alpha) \right| \leq \theta(k) \left( 1-\frac{1}{p_i} \right) 2\sqrt{q} W(k).
\]
The combination of \eqref{eq:n=2_ip(1)}, \eqref{eq:siev_ip(0)}, \eqref{eq:siev_ip(1)} and the above bounds yields the desired result.

The $q\equiv 3\pmod{4}$ case follows similarly, but with \eqref{eq:n=2_ip(2)} in mind.
\end{proof}
We are now ready to proceed with our existence results.
We start with the simplest condition to check, which follows from \eqref{eq:both} and the fact that $W(q_2)=W(q^2-1)/2$, namely
\[ \sqrt{q} \geq 2W(q^2-1) . \]
Now, recalling that $8|(q^2-1)$ amd  apply the bound of for $c_6(t)$ in Lemma~\ref{lem:w(r)}, with $t=(q^2-1)/4$ to $W(q_0) =W((q^2-1)/4)$  we deduce that a sufficient condition is 
\[ q > (2\cdot46.103)^6 /4 \simeq 1.536 \cdot 10^{11}.\]
So assume $q<1 .537 \cdot 10^{11}$ and so $q^2-1 < 2.363\cdot 10^{22}$.   It follows that $\nu(q^2-1) \leq 17$ and so $\nu(q_2) \leq 16$.  For the moment suppose also $\nu(q^2-1) \ge 11$. We proceed to employ the criterion \eqref{eq:both} with $\nu(k) = 2$ and $s \le 14$ so that $\varepsilon \geq 1- \frac{1}{7}- \frac{1}{11} \ - \cdots - \frac{1}{59} \geq 0.335869$, the sum relating to the 14 primes between 7 and 59.  Then \eqref{eq:both} holds whenever $q >488500$.  Now assume $q<488500$ which implies that $q^2-1 <2.3863\cdot 10^{11}$.  Hence $q^2-1 16893\cdot 10^9$ and  $\nu(q^2-1) \le10$. Suppose  $\nu(q^2-1)=10$ and  apply the sieve with  $\nu(k)=1, s=8$.  The condition holds if $q>41101$.  So assume $q < 41101$ so that $q^2-1 \leq 1.6893 \times 10^9$.   This yields $\nu(q^2-1) \leq 9$.    Again apply the sieve with $\nu(k)=1, s \leq 7$.  The condition then holds if $q> 25457$.  A final round of sieving with $q< 25457$, $q^2-1 < 6.4803 \cdot  10^8$   yields success if $q>14850$.  So we can suppose $q<14850$, $q^2-1 <2.2053\cdot  10^8$.
In the interval $3\leq q\leq 14850$, there are exactly $1784$ odd prime powers and we first attempt to use Proposition~\ref{prop:n=2(1)}. A quick computation reveals that, in the interval in question, there are exactly $744$ odd prime powers, where \eqref{eq:n=2_ip(1)} or \eqref{eq:n=2_ip(2)}, accordingly, do not hold, with all the mentioned quantities explicitly computed, with $q=14821$ being the largest among them.

Then, we move on to the sieving part, i.e., Theorem~\ref{thm:siev4}. Namely, we attempt to satisfy the conditions of Theorem~\ref{thm:siev4} as follows. Until we run out of prime divisors of $k$, or until $\varepsilon\leq 0$, we always add to the set of sieving primes, that is, the primes $p_1,\ldots ,p_s$ in Theorem~\ref{thm:siev4}, the largest prime divisor not already contained in the set. If, for one such set of sieving primes, the condition of Theorem~\ref{thm:siev4} is valid, then
the desired result holds for the prime power in question.

This procedure was successful, for most of the aforementioned $744$ prime powers. The $101$ exceptional prime powers, for which this procedure failed are listed in the $n=2$ line of  Table~\ref{tab:exc}.
%

So, to sum up our results so far,  we have proved the following.
\begin{theorem}\label{thm:main_n=2}
For every odd prime power $q$ not listed in Table~\ref{tab:exc} and $\beta\in\F_{q}^*$, there exists some  $2$-primitive $\xi\in\F_{q^2}$ such that $\Tr(\xi)=\beta$.
\end{theorem}
\section{Completion of the proof of Theorem~\ref{thm:main}} \label{sec:completion}
Then we move on to an explicit verification for the remaining possible exceptions, that is the pairs of Table~\ref{tab:exc}. For this purpose,  for all the corresponding pairs $(q,n)$, we check whether the set of the traces of the $2$-primitive elements of $\F_{q^n}$ coincides with $\F_q^*$, when $n=2$, and with $\F_q$, when $n= 3$. This test required about 3-4 minutes of computer time in a modern mid-range laptop.

\begin{table}[t] \footnotesize
\begin{center}
\begin{tabular}{|l|p{0.8\textwidth}|l|} \hline
$n$ & $q$ & \# \\ \hline\hline
$2$ & $3$, $5$, $7$, $9$, $11$, $13$, $17$, $19$, $23$, $25$, $27$, $29$, $31$, $37$, $41$, $43$, $47$, $49$, $53$, $59$, $61$, $67$, $71$, $73$, $79$, $81$, $83$, $89$, $97$, $101$, $103$, $109$, $113$, $121$, $125$, $127$, $131$, $137$, $139$, $149$, $151$, $157$, $169$, $173$, $181$, $191$, $197$, $199$, $211$, $229$, $239$, $241$, $269$, $281$, $307$, $311$, $331$, $337$, $349$, $361$, $373$, $379$, $389$, $409$, $419$, $421$, $461$, $463$, $509$, $521$, $529$, $569$, $571$, $601$, $617$, $631$, $659$, $661$, $701$, $761$, $769$, $841$, $859$, $881$, $911$, $1009$, $1021$, $1231$, $1289$, $1301$, $1331$, $1429$, $1609$, $1741$, $1849$, $1861$, $2029$, $2281$, $2311$, $2729$, $3541$ & 101 \\ \hline
$3$ & $5$, $9$, $13$, $25$, $37$, $49$, $121$ &  \;\; 7 \\ \hline\hline
\multicolumn{2}{|r|}{\textbf{Total:}} & 108 \\ \hline
\end{tabular}
\end{center}
\caption{Pairs $(q,n)$ for which the existence of $2$-primitive elements with prescribed trace was not dealt with theoretically.\label{tab:exc}}
\end{table}
The computations validated all the existence claims in Theorem~\ref{thm:main} for all the pairs $(q,n)$ of Table~\ref{tab:exc}  with the exception, when $n=2$, of $q=3,5,7,9,11,13$ and $31$,  these being genuine exceptions. In particular, they were successful  for all pairs $(q,n)$ with $n= 3$.  Finally, for the exceptions we present the possible traces of $2$-primitive elements in Table~\ref{tab:n=2_traces}. This completes the proof of Theorem~\ref{thm:main}.
\begin{table}[hbt]\footnotesize
\begin{center}
\begin{tabular}{|l|p{0.55\textwidth}|l|} \hline
$q$ & Traces & \# \\ \hline\hline
$3$ & $0$ & 1\\ \hline
$5$ & $2$, $3$ & 2\\ \hline
$7$ & $1$, $2$, $5$, $6$ & 4\\ \hline
$9^*$ & $\pm 1$, $ \pm i$ & 4 \\ \hline
$11$ & $1$, $2$, $3$, $4$, $7$, $8$, $9$, $10$ & 8 \\ \hline
$13$ & $1$, $3$, $4$, $5$, $6$, $7$, $8$, $9$, $10$, $12$ & 10 \\ \hline
$31$ & $1$, $2$, $3$, $4$, $5$, $6$, $7$, $8$, $9$, $10$, $12$, $13$, $14$, $15$, $16$, $17$, $18$, $19$, $21$, $22$, $23$, $24$, $25$, $26$, $27$, $28$, $29$, $30$ & 28 \\ \hline
\multicolumn{3}{l}{\rule{0em}{1.2em}\textbf{*} For $q=9$, $i$ is a root of $X^2+1 \in\F_3[X]$}
\end{tabular}
\end{center}
\caption{Traces of $2$-primitive elements of $\F_{q^2}$ for $q=3,5,7,9,11,13$ and $31$.\label{tab:n=2_traces}}
\end{table}


\end{document}